\documentclass[12pt]{amsart}
\usepackage{amsmath,amsthm,amsfonts,amssymb,latexsym,enumerate,hyperref,thmtools}

\begin{document}

\theoremstyle{plain}

\newtheorem{thm}{Theorem}[section]
\newtheorem{lem}[thm]{Lemma}
\newtheorem{conj}[thm]{Conjecture}
\newtheorem{pro}[thm]{Proposition}
\newtheorem{cor}[thm]{Corollary}
\newtheorem{que}[thm]{Question}
\newtheorem{rem}[thm]{Remark}
\newtheorem{defi}[thm]{Definition}

\newtheorem*{thmA}{THEOREM A}
\newtheorem*{thmB}{THEOREM B}
\newtheorem*{CorB}{COROLLARY B}
\newtheorem*{thmC}{THEOREM C}

\newtheorem*{thmAcl}{Main Theorem$^{*}$}
\newtheorem*{thmBcl}{Theorem B$^{*}$}

\numberwithin{equation}{section}

\newcommand{\Maxn}{\operatorname{Max_{\textbf{N}}}}
\newcommand{\Syl}{\operatorname{Syl}}
\newcommand{\dl}{\operatorname{dl}}
\newcommand{\Con}{\operatorname{Con}}
\newcommand{\cl}{\operatorname{cl}}
\newcommand{\Stab}{\operatorname{Stab}}
\newcommand{\Aut}{\operatorname{Aut}}
\newcommand{\Inn}{\operatorname{Inn}}
\newcommand{\Out}{\operatorname{Out}}
\newcommand{\Ker}{\operatorname{Ker}}
\newcommand{\pcore}{\mathbf{O}}
\newcommand{\fl}{\operatorname{fl}}
\newcommand{\Irr}{\operatorname{Irr}}
\newcommand{\SL}{\operatorname{SL}}
\newcommand{\GL}{\operatorname{GL}}
\newcommand{\SU}{\operatorname{SU}}
\newcommand{\GU}{\operatorname{GU}}
\newcommand{\Sp}{\operatorname{Sp}}
\newcommand{\Spin}{\operatorname{Spin}}
\newcommand{\PGL}{\operatorname{PGL}}
\newcommand{\PSL}{\operatorname{PSL}}
\newcommand{\PSU}{\operatorname{PSU}}
\newcommand{\PGU}{\operatorname{PGU}}
\newcommand{\POmega}{\operatorname{P\Omega}}
\newcommand{\Sz}{\operatorname{Sz}}
\newcommand{\FF}{\mathbb{F}}
\newcommand{\NN}{\mathbb{N}}
\newcommand{\N}{\mathbf{N}}
\newcommand{\C}{\mathbf{C}}
\newcommand{\OO}{\mathbf{O}}
\newcommand{\F}{\mathbf{F}}

\renewcommand{\labelenumi}{\upshape (\roman{enumi})}

\providecommand{\St}{\mathsf{St}}
\providecommand{\E}{\mathbf{E}}
\providecommand{\PSp}{\mathrm{PSp}}
\providecommand{\Sp}{\mathrm{Sp}}
\providecommand{\SO}{\mathrm{SO}}
\providecommand{\re}{\mathrm{Re}}

\def\irrp#1{{\rm Irr}_{p'}(#1)}

\def\Z{{\mathbb Z}}
\def\C{{\mathbb C}}
\def\Q{{\mathbb Q}}
\def\sbs{\subseteq}
\def\irr#1{{\rm Irr}(#1)}
\def\zent#1{{\bf Z}(#1)}
\def\cent#1#2{{\bf C}_{#1}(#2)}
\def\syl#1#2{{\rm Syl}_#1(#2)}
\def\nor{\triangleleft\,}
\def\oh#1#2{{\bf O}_{#1}(#2)}
\def\Oh#1#2{{\bf O}^{#1}(#2)}
\def\zent#1{{\bf Z}(#1)}
\def\det#1{{\rm det}(#1)}
\def\ker#1{{\rm ker}(#1)}
\def\norm#1#2{{\bf N}_{#1}(#2)}
\def\alt#1{{\rm Alt}(#1)}
\renewcommand{\Im}{{\mathrm {Im}}}
\newcommand{\Ind}{{\mathrm {Ind}}}
\newcommand{\diag}{{\mathrm {diag}}}
\newcommand{\soc}{{\mathrm {soc}}}
\newcommand{\End}{{\mathrm {End}}}
\newcommand{\sol}{{\mathrm {sol}}}
\newcommand{\Hom}{{\mathrm {Hom}}}
\newcommand{\Mor}{{\mathrm {Mor}}}
\newcommand{\Mat}{{\mathrm {Mat}}}
\newcommand{\Tr}{{\mathrm {Tr}}}
\newcommand{\tr}{{\mathrm {tr}}}
\newcommand{\Gal}{{\it Gal}}
\newcommand{\Spec}{{\mathrm {Spec}}}
\newcommand{\ad}{{\mathrm {ad}}}
\newcommand{\Sym}{{\mathrm {Sym}}}
\newcommand{\Char}{{\mathrm {char}}}
\newcommand{\pr}{{\mathrm {pr}}}
\newcommand{\rad}{{\mathrm {rad}}}
\newcommand{\abel}{{\mathrm {abel}}}
\newcommand{\codim}{{\mathrm {codim}}}
\newcommand{\ind}{{\mathrm {ind}}}
\newcommand{\Res}{{\mathrm {Res}}}
\newcommand{\Ann}{{\mathrm {Ann}}}
\newcommand{\Ext}{{\mathrm {Ext}}}
\newcommand{\Alt}{{\mathrm {Alt}}}
\newcommand{\CC}{{\mathbb C}}
\newcommand{\ch}{{\mathcal C}}
\newcommand{\CB}{{\mathbf C}}
\newcommand{\RR}{{\mathbb R}}
\newcommand{\QQ}{{\mathbb Q}}
\newcommand{\ZZ}{{\mathbb Z}}
\newcommand{\NB}{{\mathbf N}}
\newcommand{\ZB}{{\mathbf Z}}
\newcommand{\OB}{{\mathbf O}}
\newcommand{\EE}{{\mathbb E}}
\newcommand{\PP}{{\mathbb P}}
\newcommand{\GC}{{\mathcal G}}
\newcommand{\HC}{{\mathcal H}}
\newcommand{\AC}{{\mathcal A}}
\newcommand{\SC}{{\mathcal S}}
\newcommand{\TC}{{\mathcal T}}
\newcommand{\CL}{{\mathcal C}}
\newcommand{\EC}{{\mathcal E}}
\newcommand{\Om}{\Omega}
\newcommand{\eps}{\epsilon}
\newcommand{\varep}{\varepsilon}
\newcommand{\al}{\alpha}
\newcommand{\chis}{\chi_{s}}
\newcommand{\sigmad}{\sigma^{*}}
\newcommand{\PA}{\boldsymbol{\alpha}}
\newcommand{\gam}{\gamma}
\newcommand{\lam}{\lambda}
\newcommand{\la}{\langle}
\newcommand{\ra}{\rangle}
\newcommand{\hs}{\hat{s}}
\newcommand{\htt}{\hat{t}}
\newcommand{\tn}{\hspace{0.5mm}^{t}\hspace*{-0.2mm}}
\newcommand{\ta}{\hspace{0.5mm}^{2}\hspace*{-0.2mm}}
\newcommand{\tb}{\hspace{0.5mm}^{3}\hspace*{-0.2mm}}
\def\skipa{\vspace{-1.5mm} & \vspace{-1.5mm} & \vspace{-1.5mm}\\}
\newcommand{\tw}[1]{{}^#1\!}
\renewcommand{\mod}{\bmod \,}

\title{Characters, Commutators and Centers of Sylow Subgroups}

\author{Gabriel Navarro}
\address{Department of Mathematics,   Universitat de Val\`encia, 46100 Burjassot,
Val\`encia, Spain}
\email{gabriel@uv.es}
\author{Benjamin Sambale}
\address{Institut f\"ur Algebra, Zahlentheorie und Diskrete Mathematik, Leibniz Universit\"at Hannover, Welfengarten~1, 30167 Hannover, Germany}
\email{sambale@math.uni-hannover.de}
\date{\today}
 
\thanks{
The research of the first author is supported by the Ministerio de Ciencia e Innovaci\'on PID2019-103854GB-I00. 
The work on this paper started with a visit of the second author in Valencia in February 2022. He thanks the first author for his great hospitality.
The second author is supported by the German Research Foundation (projects SA 2864/1-2 and SA 2864/3-1). Both authors thank Gunter Malle
for a careful reading of the manuscript and for discussions leading to settle the case $p=3$ of Theorem A. We also appreciate some computations provided by Jianbei An and some useful remarks by Robert Guralnick.}



\begin{abstract}
The character table of a finite group $G$ determines whether $|P:P'|=p^2$ and whether $|P:\zent P|=p^2$,
where $P$ is a Sylow $p$-subgroup of $G$. To prove the latter,
we give a detailed classification of those groups in terms of the generalized
Fitting subgroup.
\end{abstract}

\maketitle
 
\section{Introduction} 
Richard Brauer's Problem 12 from \cite{B} is a source of inspiration for
discovering interactions between global and local properties
of a finite group. Brauer asked what properties of a Sylow $p$-subgroup $P$ of a finite group $G$ can be detected by the irreducible complex
characters of $G$. One of the main objectives was to find if the irreducible characters of $G$ \textit{knew} whether
$P$ is abelian. This was first settled in the affirmative in \cite{KimmerleSandling}, then in \cite{NST2} with a precise algorithm, and finally in \cite{MN} (and \cite{KM}) with the
solution of Brauer's Height Zero conjecture for principal  blocks. 
The next natural step is to study how the characters of $G$
affect $P/P'$, where $P'=[P,P]$ is the commutator subgroup of $P$, and vice versa.  It was discovered in \cite{Ma} that
the character table of a $p$-group $P$ does not even determine $|P^{\prime\prime}|$, where $P^{\prime\prime}=[P^\prime,P^\prime]$,
answering in the negative a  question of Brauer.
The group $P/P'$ is an object of interest if only by the McKay conjecture, that asserts that
 the number of irreducible characters of $G$ of degree not divisible by $p$, equals the number
of conjugacy classes of the semidirect product of $P/P'$ with $\norm GP/P$.
In fact, the Galois version of the McKay conjecture \cite{NG} predicts that the exponent of the group $P/P'$ is known by the
characters of $G$, and indeed, this question was reduced to simple groups in \cite{NT} and proved in \cite{M} for $p=2$.
The non-abelian $2$-groups with $|P/P'|=2^2$ are the well-known maximal class 2-groups: dihedral, semidihedral and generalized quaternions.
(For  odd primes $p$,  there are many
$p$-groups $P$ with $|P/P'|=p^2$.) The character table of $G$ does detect whether a Sylow $2$-subgroup $P$ of $G$ is in this class (see \cite{NST} and the related result \cite{NRSV}).  For $p=3$, it was discovered in \cite{NST} that the Alperin--McKay conjecture indeed implies that
$|P/P'|=3^2$ if and only if the number of irreducible characters of degree not divisible by $3$ in the principal $3$-block of $G$ is $6$ or $9$. 
In general, it has remained a challenge to see if the character table of $G$ determines if $|P/P'|=p^2$.
\medskip

In this paper, we solve this problem in the affirmative.

\begin{thmA}
Let $G$ be a finite group, let $p$ be a prime, and let $P$ be a Sylow $p$-subgroup of $G$.
Then the character table of $G$ determines whether $|P:P'|=p^2$.
\end{thmA}

Theorem A can be reformulated as follows: if $G$ and $H$ are finite groups with the same
character table, $P \in \syl p G$ and $Q \in \syl p H$, then $|P:P'|=p^2$ if and only if $|Q:Q'|=p^2$.
Note that if $|P:P'|=p^2$, then the isomorphism type of $P/P'$ is determined as well, because $P/P'$ is elementary abelian unless $P$ is cyclic of order $p^2$. 

\medskip
Aside from the commutator subgroup of $P$, it is natural to turn our attention to the
center $\zent P$ of $P$. Recall that $P$ is non-abelian if and only if $|P:\zent P|\ge p^2$. 
The quantity $|P:\zent P|$ is naturally related to characters. For instance,
a conjecture of G. R. Robinson~\cite{RobConjecture} (proved recently for $p>2$ in \cite{FLLMZ}) asserts that 
$|P:\zent P|$ bounds the heights of the irreducible characters of $G$.

\medskip
 In our second main result,
we prove the following.

\begin{thmB}
Let $G$ be a finite group, let $p$ be a prime, and let $P$ be a Sylow $p$-subgroup of $G$.
Then the character table of $G$ determines if $|P: \zent P|=p^2$.
\end{thmB}

To prove Theorem B, we pin down the structure of $G$ in terms of composition factors under the assumption $|P:\ZB(P)|=p^2$ (see \autoref{center} below).
This extends the celebrated structure theorem on groups with abelian Sylow $p$-subgroups. 
\medskip

The question on whether the character table of a finite group $G$ determines $|P'|$ or $|\zent P|$ seems very difficult to solve.

\medskip
Finally, we discuss the algorithm to detect whether $|P:P'|=p^2$ from the character table, which leads to some interesting problems.
  Recall that the characters of $G$ determine the characters of $G/N$ whenever
 $N\nor G$, and therefore, when studying properties of a Sylow $p$-subgroup of $G$ and characters,
we may always assume that $\oh{p'} G$, the largest normal subgroup of $G$ of order not divisible by
$p$, is trivial. As usual, we denote by ${\rm Irr}_{p'}(G)$ to be the set of the irreducible complex characters of $G$ of degree not divisible by $p$. As we shall prove,  the following holds.

\begin{thmC}
Let $G$ be a finite group, $p$ a prime, $P \in \syl pG$, with $\oh{p'}G=1$.
Let $K$ be the intersection of the kernels of the irreducible
characters in ${\rm Irr}_{p'}(G)$, and write $\bar G=G/K$.
Assume that $G$ is not almost simple. Then $|P:P'|=p^2$ if
and only if there is a $p$-element $\bar x \in \bar G$ such that $|\cent{\bar G}{\bar x}|_p=p^2$.
\end{thmC}

In the situation of Theorem~C, we shall prove that
 $K\le P'$ and $P/K\in\Syl_p(\bar{G})$ has maximal nilpotency class (see \autoref{kernels} and \autoref{TK} below).
But, in general, if $P$ has maximal class (and in particular $|P:P'|=p^2$), we do not necessarily find some $x\in P$ such that $|\CB_G(x)|_p=p^2$. 
Counterexamples to this are $\SL(2,9)$ for $p=2$, $\SL(3,19)$ for $p=3$ and $\SL(p,q)$ for $p\ge 5$ where $q-1$ is divisible by $p$ just once. This phenomenon is explained by the existence of \emph{pearls} in fusion systems (see \cite{GP}).

Since the character table of $G$ determines whether $G$ is almost simple (see \autoref{Kimmerle})
and, in this case, the isomorphism type of ${\rm soc}(G)$, we shall use ad-hoc arguments to settle 
this situation. Interestingly, we have detected that something might be occurring for this class of groups.
If $\irrp{B_0(G)}$ is the set of irreducible characters of degree not divisible
by $p$ in the principal $p$-block of $G$ almost simple, is it true that 
  $|P/P'|=p^2$ if and only if $|\irrp{B_0(G)}| \le p^2$?  At the time of this writing, 
  we do not know the answer to this question.

\section{Preliminaries on \texorpdfstring{$p$}{p}-Groups}

In this section we collect some general results on $p$-groups, which we shall later use.

\begin{lem}\label{TK}\hfill
\begin{enumerate}[(i)]
\item Let $P$ be a Sylow $p$-subgroup of the finite group $G$. Then $G'\cap\ZB(G)\cap P\le P'$.
\item Let $P$ be a finite $p$-group. Then $|P'|=p$ if and only if $\max\{|P:\CB_P(x)|:x\in P\}=p$. In particular, $|P:\ZB(P)|\le p^2$ implies $|P'|\le p$.
\item A non-abelian $p$-group $P$ has maximal nilpotency class if and only if there exists $x\in P$ such that $|\CB_P(x)|=p^2$.
\end{enumerate}
\end{lem}

\begin{proof}
These statements can be found in Huppert~\cite[Satz~IV.2.2, Aufgabe~III.24b, Satz~III.14.23]{H}. 
\end{proof}

\begin{cor}\label{lem16}
Let $P$ be a $p$-group of order $p^4$. 
\begin{enumerate}[(i)]
\item We have that $|P'|=p$ if and only if $|P:\ZB(P)|=p^2$. 
\item We have that $|P:P'|=p^2$ if and only if $P$ has maximal nilpotency class.
\end{enumerate}
\end{cor}
\begin{proof}
By \autoref{TK}, $|P:\ZB(P)|=p^2$ implies $|P'|=p$. Suppose by way of contradiction that $|P'|=|\ZB(P)|=p$. Then $P'=\ZB(P)$ is the intersection of the centralizers of $P$. By \autoref{TK}, the centralizers of non-central elements are maximal subgroups. Hence, $P'=\Phi(P)$, but there is no extraspecial group of order $p^4$.

Suppose that $|P:P'|=p^2$. Then, by \autoref{TK}, there exists $x\in P$ with $|\CB_P(x)|=p^2$ and consequently $P$ has maximal nilpotency class. Conversely, every group of maximal class satisfies $|P:P'|=p^2$.
\end{proof}

\begin{lem}\label{lemhupp}
Let $P$ be a $p$-group with a normal subgroup $N$ of index $p$. Then $|P:P'|=p^2$ if and only if $|\CB_{N/N'}(x)|=p$ for one (hence for every) $x\in P\setminus N$. 
\end{lem}
\begin{proof}
Note that $N'\unlhd P$ and $|P:P'|=|P/N':P'/N'|=|P/N':(P/N')'|$. We may therefore assume that $N$ is abelian. 
Now the claim follows from \cite[Lemma~4.6]{I}.
\end{proof}

\section{Theorem A and Non Almost Simple Groups}

In this Section we start the proof of Theorem A for non almost simple groups.
Recall that the core of a subgroup $H\le G$ is defined by $\mathrm{core}_G(H)=\bigcap_{g\in G}H^g$ where $H^g=g^{-1}Hg$. 

\begin{thm}\label{structure}
Let $G$ be a finite group, $p$ a prime, $P \in \syl pG$.
Assume that $\oh{p'} G=1={\rm core}_G(P')$.
If  $|P/P'|=p^2$ and $P$ is not abelian, then one of the following occurs:
\begin{enumerate}[(a)]
\item
$G$ has a minimal normal subgroup $N$ such that $|G/N|_p=p$. Moreover, the simple components of $N$ have cyclic Sylow $p$-subgroups.
\item
$G$ is almost simple.
\end{enumerate} 
\end{thm}

\begin{proof}
Assume that  $G$ is not almost simple.

First, we claim that we may assume that $\Oh{p'}G=G$. Indeed, assume that $K=\Oh{p'}G<G$ (recall that $K$ is the smallest normal subgroup with $p'$-index).
Notice that $P\sbs K$ and $\oh{p'}K=1$. Since $G=K\norm GP$, we also have $E={\rm core}_K(P') \nor G$ and $E=1$.
Suppose that $K$ is almost simple and $S=K^\infty \nor G$. Then $\cent KS=1$ and $\cent GS \nor G$ is a $p'$-group. Hence, $\cent GS=1$ and $G$ would be almost simple.  
Now by induction, we have that $K$ has a minimal normal subgroup $L$ such that $|G:L|_p=|K/L|_p=p$. 
Moreover, the simple components of $L$ have cyclic Sylow $p$-subgroups.
If $|L|_p=p$, then $|G|_p=p^2$, and $G$ has abelian Sylow $p$-subgroups,
against our assumption.
Let $g \in G$. Then $L^g$ is a minimal normal subgroup
of $K$. If $L^g \ne L$, then $L \cap L^g=1$ and $p^2$ would divide $|K/L|$, which cannot. Therefore $L^g=L$ for all $g \in G$, and we see that $L$ is a minimal normal subgroup of $G$ too.

Let $N$ be a minimal normal subgroup of $G$. Note that
\[|PN/N:(PN/N)'|=|P/P\cap N:(P/P\cap N)'|=|P:P'(P\cap N)|\le|P:P'|=p^2.\]
If $|P:P'(P\cap N)|=p^2$, then $P\cap N\le P'$. By Tate's theorem, $N$ has a normal $p$-complement. 
Thus $N$ is a $p$-group and $N \sbs P'$. This is not possible. 
If $G/N$ is a $p'$-group, then $N=G$, and $N$ is simple, against the hypothesis. 

So we may assume that $PN/N$ has  order $p$. If $N$ is elementary abelian, then the simple components of $N$ are indeed cyclic.  We may therefore assume that $N$ is a direct product of isomorphic non-abelian simple groups.

If $M$ is another minimal normal subgroup, then $N\cap M=1$, and $|M|_p=p$ (because
$\oh{p'}G=1$ and $|G/N|_p=p$). Then we are done by the previous paragraph.
So we have that $N$ is the unique minimal normal subgroup of $G$. 

Write $N=T_1 \times \cdots \times T_n$, where $T_i$ are simple groups which are
$G$-conjugate. Since $G$ is not
almost-simple, we have that $n>1$. Write $Q=P\cap N=Q_1 \times \cdots \times Q_n$, where $Q_i=T_i\cap P$. By way of contradiction, suppose that the $Q_i$ are non-cyclic.
Take $x\in P-Q$. 
Since we have that $|P:P'|=p^2$, then it follows that $|\cent {Q/Q'}x|=p$
by \autoref{lemhupp}. If $\langle x\rangle$ does not act transitively on $\{Q_1, \ldots, Q_n\}$, then we could write $Q/Q'=U \times V$ for non-trivial $x$-invariant subgroups $U$ and $V$.
Then $\cent Ux$ and $\cent Vx$ are two non-trivial subgroups of 
$\cent {Q/Q'}x$, which is impossible. Therefore, we can assume that $Q_i=(Q_1)^{x^{i-1}}$,
and that $n=p$. Since $Q_1$ is non-cyclic, so is $Q_1/Q_1'$.
Hence, there are $1\ne y,z \in Q_1$ such that the subgroups $\langle yQ_1'\rangle\ne \langle zQ_1'\rangle$ have order $p$.
Then $\langle \prod_{i=1}^p y^{x^i} Q' \rangle$ and 
$\langle \prod_{i=1}^p z^{x^i} Q' \rangle$ are two different subgroups of $\cent{Q/Q'} x$,
and this is the final contradiction.
\end{proof}

\begin{thm}\label{notalmsimp}
Let $G$ be a finite group, $p$ a prime, $P \in \syl pG$ with $|P|\ge p^2$.
Assume that $\oh{p'} G=1={\rm core}_G(P')$.
 Suppose that $G$ is not almost simple.
Then $|P/P'|=p^2$ if and only if there is a $p$-element $x \in G$ such that $|\cent Gx|_p=p^2$.
\end{thm}

\begin{proof}
Assume that there is a $p$-element $x \in G$ such that $|\cent Gx|_p=p^2$.
By replacing $x$ by some $G$-conjugate, we may assume that $x \in  P$.
Then $|\cent Px|\le p^2$. By \autoref{TK}, we have either $|P|=p^2$ or $P$ has maximal class. In any case, $|P:P'|=p^2$.

Assume now that $\oh{p'}G=1={\rm core}_G(P')$, that $G$ is not almost simple and that $|P:P'|=p^2$.
We show that there is a $p$-element $x \in G$ such that $|\cent Gx|_p=p^2$.
This is clear if $P$ is abelian, so assume that $P$ is not abelian.

By \autoref{structure}, 
let $N$ be a minimal normal subgroup of $G$
such that $PN/N$ has  order $p$ and such that the
Sylow subgroups of $N$ are abelian. Let $x \in P-N$.
Let $Q\in\syl p{G}$ such that $\cent Qx\in\syl p{\cent G x}$. Then $Q=\langle x\rangle(Q\cap N)$. By \autoref{lemhupp}, $|\cent{Q\cap N}{x}|=p$ and 
\[|\cent Gx|_p=|\cent Qx|=|\langle x\rangle\cent{Q\cap N}{x}|=p^2.\qedhere\]
\end{proof}

If $|P:P'|=p^2$, then the Alperin--McKay conjecture
(together with the $k(GV)$-theorem) implies that $|\irrp{B_0(G)}|\le p^2$ where $B_0(G)$ denotes the principal $p$-block of $G$. In the situation of \autoref{notalmsimp} we can prove this without relying on the Alperin--McKay conjecture.

\begin{thm}
Let $G$ be a finite group with a $p$-element $x\in G$ such that $|\CB_G(x)|_p\le p^2$. Then $|\irrp{B_0(G)}|\le p^2$.
\end{thm}
\begin{proof}
Let $P$ be a Sylow $p$-subgroup of $G$ containing $x$. Then $\CB_P(x)$ is a Sylow $p$-subgroup of $\CB_G(x)$. Hence, $\CB_P(x)$ is a defect group of the principal block $b_0$ of $\CB_G(x)$. By Brauer's third main theorem, $(x,b_0)$ is a $B_0$-subsection. Since $\CB_P(x)/\langle x\rangle$ is cyclic, the claim follows from \cite[Proposition~4.3]{Sambale}.
\end{proof}

\section{Character Tables}

As is well-known, the character table of $G$ determines the sizes of the conjugacy classes of $G$ (by the second orthogonality relation),
and the lattice of normal subgroups of $G$,
together with their orders. (This is obtained by looking at all the intersections of 
kernels of the irreducible characters.)  In particular, it is easy to detect from the character table the orders
of the chief factors of $G$, or if a normal subgroup is nilpotent or solvable.
(As we already said, it is not possible to know if this normal subgroup is abelian or not, by \cite{Ma}.)
Also, as we have mentioned,
 the character table of $G$  determines the character table of its factor groups.

We now recall a theorem of Higman (see \cite[Corollary~7.18]{N}) asserting that the character table of $G$ determines the prime divisors of the order of elements $x\in G$ (strictly speaking $x$ represents a column of the character table). In particular, we can tell which columns of the character table belong to $p$-elements or to $p'$-elements. 
 
To detect the non-abelian composition factors of a group is something much more subtle. In the following we denote the generalized Fitting subgroup of $G$ by $\F^*(G)=\F(G)*\E(G)$, where $\E(G)$ is the layer of $G$. 
The following collects what we shall need.

\begin{thm}\label{Kimmerle}
Let $G$ be a finite group.
\begin{enumerate}[(i)]
\item Let $S,T$ be finite simple groups and $s,t$ positive integers such that $|S|^s=|T|^t$. Then $s=t$ and one of the following holds
\begin{itemize}
\item $S\cong T$.
\item $\{S,T\}=\{A_8,\PSL(3,4)\}$.
\item $\{S,T\}=\{B_n(q),C_n(q)\}$ for some $n\ge 3$ and an odd prime power $q$.
\end{itemize}

\item Let $N$ be a minimal normal subgroup of $G$ such that $|N|=|A_8|^k$. Then $N\cong A_8^k$ if and only if there exists a $2$-element $x\in N$ such that $|G:\CB_G(x)|=105k$. 

\item Let $N$ be a minimal normal subgroup of $G$ such that $|N|=|B_n(q)|^k$. Then $N\cong B_n(q)^k$ if and only if there exists a $2$-element $x\in N$ such that $|G:\CB_G(x)|=\frac{1}{2}q^n(q^n\pm1)$ where $q^n\equiv\pm1\pmod{4}$. 

\item The character table of $G$ determines the isomorphism types of all minimal normal subgroups.

\item The character table of $G$ determines all chief series of $G$, that is, the isomorphism type of the chief factors and the order in which they appear.

\item The character table of $G$ determines the composition factors of $G$.

\item The isomorphism type of a quasisimple group is determined by its character table.

\item The character table of $G$ determines the size and the composition factors of the generalized Fitting subgroup $\F^*(G)$.
\end{enumerate}
\end{thm}
\begin{proof}\hfill
\begin{enumerate}[(i)]
\item This is a theorem of Cameron--Teague (see \cite[Theorem~6.1]{KLST}).
\item[(ii),(iii)] See \cite[Proposition~6.5, 6.6]{Kimmerle}. A weaker version is given in \cite[Lemma~1.9]{KimmerleSandling}.
\item[(iv)] This follows from (i)--(iii).
\item[(v)] This follows from (iv) by induction on $|G|$. It is also stated in \cite[Theorem~5]{KimmerleSandling}.
\item[(vi)] Every chief series can be refined to a composition series (but not every composition series arises in this way). 
\item[(vii)] Let $G$ be quasisimple. By (vi), the character table of $G$ determines the isomorphism type of the simple group $\bar{G}:=G/\ZB(G)$. 
If $\bar{G}$ has cyclic Schur multiplier, then $|G|$ is uniquely determined by its order. We may therefore assume that $\bar{G}$ is a simple group of Lie type with exceptional Schur multiplier. Recall that the exceptional part of the Schur multiplier (called $e$ in the Atlas~\cite[Table~5]{Atlas}) is a power of the defining characteristic $p$. On the other hand, the generic part of the Schur multiplier (called $d$) is always coprime to $p$. Thus, we may assume that $e$ is non-cyclic. This only leaves six exceptional groups $\bar{G}$. If $\bar{G}\in\{\PSU(6,2),\Sz(8),\POmega^+(8,2),{^2E_6(2)}\}$, then $e$ is a Klein four-group and its involutions are permuted transitively by an outer automorphism of order $3$. It follows that there is only one quasisimple group up to isomorphism for each possible order. Now let $\bar{G}\cong\PSL(3,4)$. Then the Schur multiplier is isomorphic to $C_{12}\times C_4$ and the universal covering group can be constructed as $\mathtt{PerfectGroup}(967680,4)$ in \cite{GAP48}. In this way we can check that there are 14 quasisimple groups and they have distinct character tables. Finally let $\bar{G}\cong\PSU(4,3)$. Here the Schur multiplier is $M\cong C_{12}\times C_3$ and $\Out(\bar{G})\cong D_8$. The Atlas tells us that $D_8$ acts faithfully on $\pcore_3(M)\cong C_3^2$. Hence, the four subgroups of order $3$ in $M$ fall into two orbits under $D_8$. This leads to 12 quasisimple groups and the corresponding character tables are available in \cite{GAP48}. Again we check that the character tables are different.

\item[(viii)] Since the Fitting subgroup $\F(G)$ is the largest nilpotent normal subgroup, $|\F(G)|$ is determined by the character table. Recall that the layer $\E(G)$ is a central product of normal subgroups $N$ such that $N'=N$ and $N/\ZB(N)$ is a direct product of isomorphic non-abelian simple groups. It is easy to spot chief factors $N/Z\cong T_1\times\ldots\times T_n$ from the character table such that $N'=N$, $Z$ is nilpotent and $T_1\cong\ldots\cong T_n$ are non-abelian simple. It remains to decide if $\ZB(N)=Z$. 
This is obvious if $Z=1$. Hence, let $Z\ne 1$.
Since the isomorphism type of $T_1$ is known, so is the Schur multiplier $M(N/Z)\cong M(T_1)^n$. Let $Z/W$ be another chief factor of $G$. Then $Z/W$ is elementary abelian, say $Z/W\cong C_p^r$ for some prime $p$. 
As in (vii) we see that the Sylow $p$-subgroup of $M(T_1)$ is cyclic, apart from finitely many cases where $p\le 3$ and the $p$-rank of $M(T_1)$ is $2$. If $r>2n$, then $\zent N<Z$ and we are done. Now let $r\le 2n$. Suppose that some $T_i$ acts non-trivially on $Z/W$. Since $T_1,\ldots,T_n$ are conjugate in $G$, they all act non-trivially, i.\,e. $N/Z$ acts faithfully on $Z/W$. Since the minimal degree of a faithfully representation of $N/Z$ over $\FF_p$ is at least $2n$, it follows that $r=2n$ and therefore $p\le 3$. 
However, $T_1$ cannot embed into the solvable group $\GL(2,p)$. This contradiction shows that $N/Z$ acts trivially on $Z/W$. We can now consider the action of $N/W$ on the next lower chief factor $W/W_1$. It is well-known that a Schur cover of $N/W$ is also a Schur cover of $N/Z$. Hence, $|W/W_1|$ is bounded in the same way as $Z/W$. So we can figure out whether $N/W$ acts trivially on $W/W_1$. Continuing in this way shows whether $N/Z$ acts trivially on $Z$ and in this case $Z=\ZB(N)$ and $N\le\E(G)$.
\qedhere
\end{enumerate}
\end{proof}

\section{Theorem A and Almost Simple Groups}

In this section we provide the tools for the proof of Theorem~A.

As is customary, we adopt the notation 
\[\GL^\epsilon(k,q^f):=\begin{cases}
\GL(k,q^f)&\text{if }\epsilon=1,\\
\GU(k,q^f)\le\GL(k,q^{2f})&\text{if }\epsilon=-1
\end{cases}\]
and similarly, $\SL^\epsilon(k,q^f)$, $\PSL^\epsilon(k,q^f)$ (here $\GU$ stands for the general unitary group).

\begin{lem}\label{lemPSL}
Let $p\ne q$ be primes such that $p>2$. Let $S:=\PSL^\epsilon(k,q^f)$ where $p\mid\gcd(k,f,q^f-\epsilon)$. 
Let $S\le G\le\Aut(S)$ such that $|G:S|_p=p$. Let $P$ be a Sylow $p$-subgroup of $G$. Then $|P:P'|=p^2$ if and only if $k=p$ and
there exists a $q$-element $s\in S$ such that $|\CB_G(s)|$ is not divisible by $p$.
\end{lem}
\begin{proof}
We start by constructing a Sylow $p$-subgroup $Q$ of $S$ following \cite{Weir}. Note first that
\[|S|=\frac{q^{fk(k-1)/2}}{\gcd(k,q^f-\epsilon)}\prod_{i=2}^k(q^{fi}-\epsilon^i).\]
Let $p^r$ be the largest $p$-power dividing $q^f-\epsilon$. Since $q^{f/p}\equiv q^f\equiv \epsilon\pmod{p}$, we have $r\ge 2$. 
Note that the $p$-part of $q^{fi}-\epsilon^i$ depends only on $r$ and $i$. 
If $\epsilon=1$, let $Y\le\FF_{q^f}^\times$ be of order $p^r$. 
If $\epsilon=-1$, take $Y\le\FF_{q^{2f}}^\times$ of order $p^r$.
Let $Y_1\le Y$ be of order $\gcd(k,p^r)$.
The diagonal matrices in $\SL^\epsilon(k,q^f)$ with entries in $Y$ can be realized by $k$-tuples
\[D:=\{(y_1,\ldots,y_k)\in Y^k:y_1\ldots y_k=1\}.\]
Modulo scalars, we obtain 
\[\overline{D}:=D/\langle (y,\ldots,y):y\in Y_1\rangle.\]
We denote the elements in $\overline{D}$ by $\overline{(y_1,\ldots,y_k)}$ as usual. Finally, let $L$ be a Sylow $p$-subgroup of the symmetric group $S_k$. Then $L$ acts on $\overline{D}$ by permuting the coordinates and $Q\cong \overline{D}\rtimes L$ (it can be checked that $Q$ has indeed the correct order). Recall that $L$ is a direct product of iterated wreath products corresponding to the $p$-adic expansion of $k$. 

By the Atlas~\cite[Table~5]{Atlas}, we have 
\[G/S\le\Out(S)\cong \begin{cases}
C_d\rtimes (C_f\times C_2)&\text{if }\epsilon=1,\\
C_d\rtimes C_{2f}&\text{if }\epsilon=-1
\end{cases}\]
where $d=\gcd(k,q^f-\epsilon)$ denotes the order of the diagonal automorphism group, $C_f$ or $C_{2f}$ stands for the field automorphism group and the graph automorphism group $C_2$ acts by inversion on $C_d$ (if $\epsilon=1$).
Let $\gamma\in\Out(S)$ be a diagonal automorphism induced by $\diag(\zeta,1,\ldots,1)\in\GL^\epsilon(k,q^f)$ for some $\zeta\in Y$ such that $\gamma^p$ is an inner automorphism of $S$. Let $\delta$ be the field automorphism $\lambda\mapsto\lambda^{q^{f/p}}$ for $\lambda\in\FF_{q^f}$ (respectively $\lambda\mapsto\lambda^{q^{2f/p}}$ for $\lambda\in\FF_{q^{2f}}$ if $\epsilon=-1$).
Then $P$ induces an automorphism $\alpha=\gamma^i\delta^j$ on $Q$ with $0\le i,j\le p-1$. Observe that $\alpha$ normalizes $\overline{D}$ and acts trivially on $Q/\overline{D}\cong L$. Moreover $\gamma^p$ corresponds to some element of $\overline{D}$. Hence, $P/\overline{D}\cong L\times C_p$. Now $|P:P'|=p^2$ forces $|L|=p$ by \cite[Satz~III.15.3]{H}, i.\,e. $k=p=|Y_1|$.
So we can assume that $Y=\langle\zeta\rangle$ and $\gamma^p$ is identified with $\overline{(\zeta^{p-1},\zeta^{-1},\ldots,\zeta^{-1})}\in\overline{D}$.
Since $D$ is generated by the elements $(\zeta,\zeta^{-1},1,\ldots,1)$, $(1,\zeta,\zeta^{-1},1,\ldots,1),\ldots,(1,\ldots,1,\zeta,\zeta^{-1})$, we have
\[Q'=\langle\overline{(\zeta,\zeta^{-2},\zeta,1,\ldots,1)},\ldots\rangle.\]
Let $t:=(\zeta,\zeta^{-1},1,\ldots,1)\in D$. We compute
\begin{align*}
\overline{t}^p&=\overline{(\zeta,\zeta^{-2},\zeta,1,\ldots,1)}^{p-1}\overline{(1,\zeta,\zeta^{-2},\zeta,1,\ldots,1)}^{p-2}\ldots\\
&\qquad\ldots\overline{(1,\ldots,1,\zeta,\zeta^{2},\zeta)}^2\overline{(\zeta,1\ldots,1,\zeta,\zeta^{-2})}\in Q'
\end{align*}
and deduce
\[Q/Q'=L\times\langle \overline{t}\rangle Q'/Q'\cong C_p\times C_p.\]
By \autoref{lemhupp}, $|P:P'|=p^2$ if and only if $|\CB_{Q/Q'}(\alpha)|=p$. 
Setting $\sigma=(1,\ldots,p)\in L$, we have 
\[\gamma(\sigma)=\sigma\overline{(\zeta^{-1},1\ldots,1,\zeta)}\equiv\sigma\overline{t}\pmod{Q'}.\]
This yields $\CB_{Q/Q'}(\gamma)=\langle \overline{t}\rangle Q'/Q'$. On the other hand, $\delta(\sigma)=\sigma$ and
$\delta(\overline{t})\equiv\overline{t}\pmod{Q'}$ since $q^{f/p}\equiv \epsilon\pmod{p}$. 
We conclude that $\CB_{Q/Q'}(\alpha)=\CB_{Q/Q'}(\gamma^i)$. 
Therefore, $|\CB_{Q/Q'}(\alpha)|=p$ if and only if $i\ne 0$, i.\,e. $P/Q=\langle\alpha\rangle$ does not induce a field automorphism. 

We now translate this condition into some character table property.
By \autoref{Kimmerle}, the character table of $G$ determines the isomorphism type of $S$ and in turn also $k$ and $q^f$ (there are no exceptional isomorphisms for the given parameters). 
We aim to count $q$-elements $s\in S$ such that $|\CB_G(s)|$ is not divisible by $p$. 
Recall that every unitriangular matrix in $\SL(p,q^f)$ is similar to a matrix in $\SU(p,q^f)$ and two matrices in $\SU(p,q^f)$ are similar if and only if they are conjugate in $\GU(p,q^f)$ (see \cite[p. 34, Case (A)]{WallUnitary}). We can therefore think of $s$ as a unitriangular matrix.
Let $V$ be the $p$-dimensional vector space over $\FF_{q^f}$ ($\FF_{q^{2f}}$ if $\epsilon=-1$) corresponding to $S$. Suppose that there is a non-trivial $s$-invariant decomposition $V=U\oplus W$ and $s=s_U\oplus s_W$ such that $s_U$ acts on $U$ and $s_W$ acts on $W$. Let $d_U:=\dim U$ and $d_W:=\dim W$. 
Then $\zeta^{d_W} s_U\oplus \zeta^{-d_U}s_W\in\CB_S(s)$ has order divisible by $p$. Now assume that $s$ acts indecomposably on $V$. 
Then $s$ is similar to a single unitriangular Jordan block $s'$. 
Recall that the centralizer of $s'$ in $\GL(p,q^f)$ consists of unitriangular matrices. It follows that $\CB_S(s)$ is a $q$-group. 
Since the elements in $\CB_{\GL^\epsilon(p,q^f)}(s)$ have only one eigenvalue (with multiplicity $p$), it follows that 
\[\frac{|\GL^\epsilon(p,q^f):\CB_{\GL^\epsilon(p,q^f)}(s)|}{|\SL^\epsilon(p,q^f):\CB_{\SL^\epsilon(p,q^f)}(s)|}=
|\GL^\epsilon(p,q^f):\CB_{\GL^\epsilon(p,q^f)}(s)\SL^\epsilon(p,q^f)|=p.\]
Hence, $S$ has precisely $p$ conjugacy classes of such elements and they are represented by $\gamma^i(s)$ where $i=0,\ldots,p-1$. 
We may choose $s\in\PSL^\epsilon(p,q)$.
If $\alpha$ is a field automorphism, it commutes with $\gamma$ and fixes the $S$-conjugacy class of each $\gamma^i(s)$. 
Hence in this case, for every $q$-element $s\in S$, $|\CB_G(s)|$ is divisible by $p$. 
If, on the other hand, $\alpha$ is not a field automorphism, the elements $\gamma_i(s)$ are fused in $G$. In particular there is a $q$-element $s\in S$ (unique up to conjugation) such that $|\CB_G(s)|$ is not divisible by $p$.
\end{proof}

The two cases in \autoref{lemPSL} arise for example when $p=3$, $S\cong\PSU(3,8)$ and $|G/S|=3$. 
The relevant groups can be constructed by $\mathtt{PrimitiveGroup}(513,a)$ with $a=3,4,5$ in \cite{GAP48}. The character tables have the names \verb#U3(8).3_1#, \verb#U3(8).3_2# and \verb#U3(8).3_3# respectively.
Here $|P:P'|=9$ if and only if $a\in\{4,5\}$.

The following proof was kindly provided to us by Gunter Malle.

\begin{lem}\label{lemDE}
Let $G$ be an almost simple group with socle $S\in\{D_4(q),E_6(q),{^2E_6(q)}\}$ and $|G/S|_3=3$. 
Let $P\in\Syl_3(G)$. Then $|P:P'|>9$ or the isomorphism type of $P$ can be deduced from the character table of $G$.
\end{lem}
\begin{proof}
Let $q=p^f$ be a prime power. 
Suppose first that $S=E_6^\epsilon(q)$ (where $E_6^+(q)=E_6(q)$ and $E_6^-(q)={}^2E_6(q)$). If $3\nmid q-\epsilon$, then $P$ induces a field automorphism on $S$ and if $3\nmid f$, then $P$ induces a diagonal automorphism. This determines the isomorphism type of $P$.
Hence, we may assume that $3\mid (q-\epsilon,f)$.
Now $\Out(S)$ has Sylow 3-subgroups isomorphic to $C_3\times C_m$ where $m$ is
the 3-part of $f$. In $\Out(S)$ the graph automorphism of order~2 inverts the
diagonal automorphism and centralizes the field automorphism, so there are
three $\Out(S)$-conjugacy classes of subgroups of order~3 in $\Out(S)$,
generated by a diagonal automorphism, by a field automorphism and by their
product, respectively. Correspondingly, there are three possible candidates for
a Sylow $3$-subgroup $P$ of $G$.

Assume first that $\eps=1$. Since the adjoint
group of type $E_6$ has a 3-element $s$ with disconnected centralizer of type
$D_4(q).(q-1)^2.3$ lying in the derived subgroup, $S$ has three irreducible
characters of degree $q^9\Phi_3^2\Phi_5\Phi_6\Phi_8\Phi_9\Phi_{12}/3$,
corresponding to the Steinberg character of $D_4(q)$ under Jordan decomposition.
According to Lusztig's parametrization these are the only irreducible characters
of $S$ of this degree. The field automorphisms of $S$ leave these characters
invariant, while the diagonal automorphism of order~3 permutes them
transitively. Thus, the existence of characters of degree
$bq^9\Phi_3^2\Phi_5\Phi_6\Phi_8\Phi_9\Phi_{12}/3$, with $b$ dividing
$|\Out(S)|$ and prime to~3, allows one to
identify the case when $G$ induces a field automorphism. If $G$ induces the
diagonal automorphism of order~3, it contains the adjoint group of type $E_6$.
This also contains an element of order 3 with centralizer of type $A_2(q)^3.3$,
which in turn contains a Sylow 3-subgroup $P$ of $G$. From this it is easy to
check that $|P:P'|\ge 27$. This completes the proof for the case $\eps=1$. When
$\eps=-1$ an entirely similar argument applies.

Now assume $S=D_4(p^f)$. Again, Sylow 3-subgroups of $\Out(S)$ are isomorphic to
$C_3\times C_m$ where $m$ is the 3-part of $f$. Clearly we may assume $m>1$.
Again, there are three $\Out(S)$-conjugacy classes of subgroups of order~3 in
$\Out(S)$, generated by a graph automorphism, by a field automorphism and by
their product, respectively. If $q=3^f$ is a 3-power, then a Sylow $3$-subgroup
of $S$ has commutator factor group generated by the images of the root subgroups
of order~$q$ for the four simple roots. The field automorphism stabilizes each
root subgroup, while the graph automorphism stabilizes one and interchanges the
other three. Hence in either case, $|P:P'|>9$.

Finally assume that $q$ is not a 3-power. By Lusztig's Jordan decomposition of
characters, the only irreducible characters $\chi$ of $S$ with $\chi(1)_p=q^6$
and $\chi(1)$ not divisible by $\Phi_4^2$ are three unipotent characters of
degree $q^6\Phi_3\Phi_6$.
These are fixed by the field automorphisms but permuted transitively by
the graph and the graph-field automorphism. Thus the existence of irreducible
characters of degree $bq^6\Phi_3\Phi_6$ with $b$ dividing $|\Out(S)|$ and prime
to~3 allows one to identify the case when $G$ induces a field automorphism.
The extension
$H$ of $S$ by the graph automorphism occurs as a subgroup of $F_4(q)$. The latter
group contains a subsystem subgroup of type $A_2(q)^2$ which in turn contains
a Sylow 3-subgroup $P$ of $F_4(q)$. The latter clearly has $|P:P'|>9$. Comparing
orders, it is also a Sylow 3-subgroup of $H$. 
\end{proof}

\section{Proof of Theorem A}

The following slightly extends a theorem of Berkovich (see \cite[Theorem 7.7]{N}).

 \begin{thm}\label{kernels}
 Let $G$ be a finite group, $p$ a prime and $P\in \syl pG$. 
 Let $$K=\bigcap_{\chi \in \irrp G} \ker\chi\, .$$
 Then $K={\rm core}_G(NP')$, where $N$ is the largest normal subgroup of $G$ such that $\cent NP=1$. 
 \end{thm}

\begin{proof}
If $M \nor G$, then notice that $\cent MP=1$ implies that $M$ is a $p'$-group.
Indeed, if $p$ divides $|M|$, then $1<Q=P\cap M \nor P$ and by elementary group
theory, $Q\cap \zent P>1$. Now, if $L,M \nor G$, and $\cent L{P}=1=\cent MP$, then $\cent{LM}P=\cent LP\cent MP=1$, by coprime action.
Therefore, there is a largest normal subgroup $N$ of $G$ such that $\cent NP=1$.
Of course, $N$ does not depend on $P$, since $\cent N{P^g}=\cent NP^g=1$ for $g \in G$.
Also, $N$ is characteristic in  $G$. If we denote $X(G):=N$, notice that $X(G/N)=1$,
by coprime action.

Let $\chi \in \irrp G$. Let $\theta \in \irr N$ be $P$-invariant under $\chi$. Since $\cent NP=1$, we have that $\theta=1_N$ by the Glauberman correspondence. Thus $N \sbs K$.
Notice that it is no loss to assume that $N=1$. 
We want to prove that $K={\rm core}_G(P')$.

By Theorem 7.7 of \cite{N}, we have that $K$ has a normal $p$-complement $R$.
Suppose that $\gamma \in \irr R$ is $P$-invariant. Then $\gamma$ has an extension $\theta \in \irr{RP}$. Since $\theta$ has $p'$-degree and $|G:RP|$ is not divisible by $p$,
then $\theta^G$ contains an irreducible character $\chi$ of $p'$-degree. Then $R \sbs \ker\chi$, and therefore $\gamma=1_R$. Therefore $\cent RP=1$, and thus $R\sbs N=1$.
We have then that $K$ is a $p$-group.  If $\chi \in \irrp G$, then $\chi_P$ has a
linear constituent $\lambda \in \irr P$. Then $\chi_{P'}$ contains $1_{P'}$, and thus ${\rm core}_G(P') \sbs \ker\chi$.  Finally, if $\lambda \in \irr P$ is linear, then $\lambda^G$
contains a $p'$-degree $\chi \in \irr G$. Thus $\chi_P=\lambda + \Delta$ for some character $\Delta$ of $P$ or $\Delta=0$.
Then ${\ker \chi} \cap P \sbs \ker\lambda$. Hence $K\cap P \sbs \ker \lambda$ for all
$\lambda$. Therefore $K=K\cap P \sbs P'$ and the theorem follows.
\end{proof}

Assuming $\oh{p'} G=1$, the proof of \autoref{kernels} shows that $K={\rm core}_G(P')$. 
Hence, the condition $\oh{p'} G=1={\rm core}_G(P')$
can be read off from the character table. Thus, in order to prove Theorem~A we may assume that $\oh{p'} G=1={\rm core}_G(P')$.
Moreover, by \autoref{notalmsimp}, we may assume that $G$ is almost simple. Now Theorem~A follows from the next result.

\begin{thm}
Let $G$ be an almost simple group with Sylow $p$-subgroup $P$. Then the character table of $G$ determines whether $|P:P'|=p^2$.
\end{thm}
\begin{proof}
For $p=2$, the claim was already shown in \cite{NST}. Thus, let $p>2$. By \autoref{Kimmerle}, the character table of $G$ determines the isomorphism type of the simple socle $S$ of $G$. Therefore, we may assume that $P\nsubseteq S$. It follows that $S$ must be a simple group of Lie type.
By the same argument as in the proof of \autoref{structure}, we may assume that $|PS/S|=p$. 
According to the Atlas~\cite[Table~5]{Atlas}, in most cases $G/S\le\Out(S)$ has a unique subgroup of order $p$ up to conjugation. In these cases $P$ is uniquely determined by $S$ and we are done. 
If $S\cong\PSL^\epsilon(k,q^f)$, then we may assume that $p\mid\gcd(k,q-\epsilon,f)$ and the claim follows from \autoref{lemPSL}. The only remaining exceptional cases are settled in \autoref{lemDE}. 
\end{proof}

Let $G=A_n$ be an alternating group with non-abelian Sylow $p$-subgroup $P$ and $|P:P'|=p^2$.
Then either $p=2$, $n\in\{6,7\}$ or $p>2$, $n=a+p^2$ with $0\le a\le p-1$. In either case $P\cong C_p\wr C_p$ has maximal nilpotency class (see \cite[Satz~III.15.3]{H}).
The sporadic groups $G$ such that $|P|\ge p^4$ and $|P:P'|=p^2$ are 
\[(G,p)\in\{(M_{11},2),(J_3,3),(Ly,5),(Co_1,5),(HN,5),(B,5),(M,7)\}\]
(this can be derived from the structure of the Sylow normalizer described in \cite{WilsonAM}). In all cases, except $G=J_3$, $P$ has maximal nilpotency class.

\section{Detecting the Center}

Let $P$ be a $p$-group.
For $Q\le P$ and $N\unlhd P$ we have $|Q:\ZB(Q)|\le|P:\ZB(P)|$ and $|P/N:\ZB(P/N)|\le|P:\ZB(P)|$ by elementary group theory. This often allows inductive arguments in the following.

\begin{lem}\label{lemcenter}
Let $G$ be a finite group with a normal $p$-subgroup $N$ such that $G/N$ has cyclic Sylow $p$-subgroups. Then the character table of $G$ determines whether $N$ is abelian.
\end{lem}
\begin{proof}
By \cite[Corollary~11.22]{Isaacs}, every $\psi\in\Irr(N)$ extends to a Sylow $p$-subgroup $P$ of the stabilizer $G_\psi$ since $P/N$ is cyclic. By \cite[Corollary~8.16]{Isaacs}, $\psi$ also extends to every Sylow $q$-subgroup of $G_\psi$ where $q\ne p$. Hence, $\psi$ extends to $G_\psi$ by \cite[Corollary~11.31]{Isaacs}. 

Now define an equivalence relation on $\Irr(G)$ by $\chi\sim\psi:\iff[\chi_N,\psi_N]\ne 0$ (note that this relation is indeed transitive). Choose representatives $\chi_1,\ldots,\chi_k\in\Irr(G)$ for each equivalence class such that $\chi_i(1)$ is as small as possible for $i=1,\ldots,k$. By Clifford theory, $(\chi_i)_N$ is a sum of $G$-conjugates of some $\psi\in\Irr(N)$. In particular, $\theta:=\chi_1+\ldots+\chi_k$ satisfies $\theta_N=\sum_{\psi\in\Irr(N)}\psi$. Since $|N|=\sum_{\psi\in\Irr(N)}\psi(1)^2$, it follows that $N$ is abelian if and only if $\theta(1)=|N|$. 
\end{proof}

The next result is taken from Gross~\cite[Theorem~C]{Gross} and Glauberman~\cite[Corollary~5]{Glauberman} respectively.

\begin{lem}\label{GG}\hfill
\begin{enumerate}[(i)]
\item Let $p>2$ be a prime and $G$ a finite group with $\pcore_{p'}(G)=1$. Let $P$ be a Sylow $p$-subgroup of $G$ and $Q:=P\cap\F^*(G)$. Then $\CB_P(Q)=\ZB(Q)$.
\item Let $G$ be a finite group with $\pcore_{2'}(G)=1=\ZB(G)$. Suppose $G$ has an abelian Sylow $2$-subgroup $P$. Let $\alpha\in\Aut(G)$ be a $2$-element that centralizes $P$. Then $\alpha$ is an inner automorphism of $G$. 
\end{enumerate}
\end{lem}

\begin{lem}\label{lemp2}
Let $G$ be a quasisimple group with a non-abelian Sylow $2$-subgroup $P$ such that $|P:\ZB(P)|=4$. Then the following holds:
\begin{enumerate}[(i)]
\item $G\cong A_7$, $3.A_7$ or $G/\ZB(G)\cong\PSL(2,q)$ for some odd prime power $q$.
\item $|P|= 8$.
\item $|\CB_{\Aut(G)}(P):\CB_{\Inn(G)}(P)|_2\le 2$ with equality if and only if $G\cong A_7$ or $G\cong\PSL(2,p^{2f})$ with $p^{2f}\equiv 9\pmod{16}$. In the latter case $\pcore_2(\CB_{\Aut(G)}(P)/\CB_{\Inn(G)}(P))$ is generated by the field automorphism $x\mapsto x^{p^f}$. 
\end{enumerate}
\end{lem}
\begin{proof}\hfill
\begin{enumerate}[(i)]
\item Suppose first that $\pcore_2(\ZB(G))\ne 1$. Then by \autoref{TK}, $\pcore_2(\ZB(G))=P'\cong C_2$. Hence, the simple group $\overline{G}:=G/\ZB(G)$ has abelian Sylow $2$-subgroup $\overline{P}$.
Those are classified by Walter's theorem. The case $\overline{G}=\PSL(2,4)\cong\PSL(2,5)$ fulfills our claim. On the other hand, $\overline{G}=\PSL(2,2^f)$ with $f\ge 3$, $\overline{G}={^2G_2(3^f)}$ or $\overline{G}=J_1$ are impossible since then $\pcore_2(\ZB(G))=1$.

Now let $\pcore_2(\ZB(G))=1$. Without loss of generality, let $\ZB(G)=1$. The simple groups with Sylow $2$-subgroup of nilpotency class $2$ were classified in \cite{GG}. We need to dismiss the last four groups mentioned there.
The Suzuki groups $\Sz(2^n)$ for $n\ge 3$ have Suzuki $2$-groups as Sylow subgroup with $|P:\ZB(P)|=|\ZB(P)|>4$. 
The group $G=\PSL(3,2)\cong\PSL(2,7)$ is on our list.
For $G=\PSL^\epsilon(3,2^n)$ with $n\ge 2$, $P$ consists of unitriangular matrices and $|P:\ZB(P)|=2^{2n}>4$. Finally, let $G=\Sp(4,2^n)$ with $n\ge 2$. If $n$ is odd, then $\Sz(2^n)\le G$ by \cite[Theorem~3.7]{Wilson}. This was already excluded above. 
If $n$ is even, then $\Sp(4,4)\le G$. By computer we check that $\Sp(4,4)$ has a Sylow $2$-subgroup $Q$ such that $|Q:\ZB(Q)|>4$. Hence, also $|P:\ZB(P)|>4$. 

\item For $\overline{G}\cong A_7$ we have $P\cong D_8$. 
Recall that $\SL(2,q)$ has (generalized) quaternion Sylow $2$-subgroups and $\PSL(2,q)$ has dihedral Sylow $2$-subgroups when $q$ is odd. In both cases $|P|=|P:\ZB(P)||\ZB(P)|=8$. We also note that $q\equiv\pm3\pmod{8}$ if $G=\SL(2,q)$ and $q\equiv \pm7\pmod{16}$ if $G=\PSL(2,q)$. 

\item If $G=A_7$, then $\Aut(G)=S_7$ and $|\CB_{S_7}(P):\CB_{A_7}(P)|=2$. If $G=3.A_7$, then $\Aut(G)=\Inn(G)$. Now let $G=\SL(2,q)$ with $q=p^f\equiv \pm3\pmod{8}$. Then $f$ is odd and $\Out(G)\cong C_2\times C_f$. 
The unique diagonal outer automorphism of order $2$ has an abelian centralizer in $G$, so it cannot fix $P$. Next, let $G=\PSL(2,q)$ with $q=p^f\equiv\pm7\pmod{16}$. Suppose first that $q\equiv 7\pmod{16}$. Then again $f$ is odd and there is only one outer automorphism $\alpha$ of order $2$. Since $q-1$ is not divisible by $4$, we may assume that $\alpha$ is the conjugation with $\diag(-1,1)$. But now $\alpha$ cannot fix an element of order $4$ in $G$. 
Finally, assume that $q=p^f\equiv -7\pmod{16}$. Here an outer diagonal automorphism $\alpha$ is induced by $\diag(\zeta,1)$ where $\zeta\in\FF_q^\times$ has order $8$. Now $\CB_G(\alpha)$ can only contain diagonal matrices. Hence, $\alpha$ does not centralize $P$.
If $f$ is odd, there are no other choices. Thus, let $f=2f'$. Since $q\equiv 9\pmod{16}$, $f'$ is odd. 
We may assume that
\[P=\langle \begin{pmatrix}
\zeta&0\\0&\zeta^{-1}
\end{pmatrix},\begin{pmatrix}
0&1\\-1&0
\end{pmatrix}\rangle/\langle -1_2\rangle.\]
Let $\beta$ be the field automorphism $x\mapsto x^{p^{f'}}$. Since $\zeta^{p^{f'}}=\zeta^{\pm3}=-\zeta^{\pm1}$, $\beta$ induces an inner automorphism on $P$. So there must be another outer automorphism of $2$-power order in the coset of $\alpha$, which centralizes $P$. \qedhere
\end{enumerate}
\end{proof}

The case $\PSL(2,9)\cong A_6$ leads to $S_6$ with Sylow $2$-subgroup $P$ and $|P:\ZB(P)|=4$. This example and $\PSL(2,q^{2f})$ were pointed out in Gross~\cite{Gross}.
In the next lemma we denote the extraspecial group of order $p^3$ and exponent $p$ by $p^{1+2}_+$.

\begin{lem}\label{lempodd}
Let $p>2$ be a prime and $G$ be a finite quasisimple group such that $|P:\ZB(P)|=p^2$ for some non-abelian Sylow $p$-subgroup $P$ of $G$. Then $P\cong p^{1+2}_+$.
\end{lem}
\begin{proof}
The extraspecial group $p^{1+2}_-$ of exponent $p^2$ cannot be a Sylow subgroup of a perfect group, since the focal subgroup would be too small. Hence, we may assume by way of contradiction that $|P|\ge p^4$. Without loss of generality, let $\pcore_{p'}(G)=1$.
Let $\mathcal{F}=\mathcal{F}_P(G)$ be the fusion system on $P$ induced by $G$. 
Following the proof of \cite[Theorem~2.1]{Oliverindexp}, we find that $\ZB(P)=P'\times A$ where $1\ne A\unlhd\mathcal{F}$. By \cite[Corollary~I.4.7]{AKO}, $A$ is strongly closed in $G$ with respect to $P$, i.\,e. $gAg^{-1}\cap P\le P$ for every $g\in G$.
By \autoref{TK}, $Z:=\ZB(G)\le P'$ and $|Z|\le p$. Then $AZ/Z$ is strongly closed in the simple group $\overline{G}:=G/Z$ with respect to the abelian Sylow subgroup $\overline{P}:=P/Z$. 

Suppose that $Z\ne 1$. Then $\overline{G}$ is not an alternating group, since $p>2$ and $|P|\ge p^4$. If $\overline{G}$ is a sporadic group, then $p=3$ and $\overline{G}\in\{Suz,McL,Fi_{22},Fi_{24}',ON,J_3\}$. Of those, only $\overline{G}\cong ON$ has abelian Sylow $3$-subgroups. But here $P$ is extraspecial of order $3^5$.
Hence, suppose that $\overline{G}$ is of Lie type. By \cite[Proposition~2.5]{FF}, $\N_{\overline{G}}(\overline{P})$ acts irreducibly on $\overline{P}/\Phi(\overline{P})$. Since $|\overline{P}:\overline{A}|=p^2$, it follows that $\ZB(P)=\Phi(P)$ and $\overline{P}$ has rank $2$. One can check from the Atlas~\cite[Table~5]{Atlas} that $Z$ does not lie in the exceptional part of the Schur multiplier (in this case $p$ would be the defining characteristic). 
If $\overline{G}=\PSL^\epsilon(d,q)$, then $p$ must divide $(d,q-\epsilon)$. This can only happen for $p=3=d$, since otherwise $\overline{P}$ has rank larger than $2$ (take diagonal matrices). Now if $9$ divides $q-\epsilon$, then $\overline{P}$ is non-abelian and otherwise $|\overline{P}|=9$. In the remaining cases we have $p=3$ and $G=E_6(q)$ or $^2E_6(q)$ by \cite[Table~5]{Atlas}. Here $F_4(q)\le G$ and $\overline{P}$ is never abelian (see \cite[Section~4.8.9]{Wilson}).

Thus, $Z=1$ and $G$ is simple.
The (simple) groups with a strongly closed subgroup are classified by Flores--Foote~\cite{FF} and the results are summarized in \cite[Theorems~II.12.12, II.12.10]{AKO}. If $S$ is of Lie rank $1$, then only $\PSU(3,p^n)$ and $^2G_2(3^n)$ for $n\ge 2$ have non-abelian Sylow $p$-subgroups of order $\ge p^4$. In both cases $|P:\ZB(P)|=p^{2n}>p^2$. The only remaining case is $G=J_3$ with $p=3$. Here $|P:\ZB(P)|=27$.
\end{proof}

By \cite[Theorem~2.1]{KimmerleSandling}, a finite group $G$ has abelian Sylow $p$-subgroups if and only if $\pcore^{p'}(G/\pcore_{p'}(G))\cong A\times S$, where $A$ is an abelian $p$-group and $S$ is a direct product of simple groups with abelian Sylow $p$-subgroups. The following theorem is a generalization.

\begin{thm}\label{center}
Let $G$ be a finite group with non-abelian Sylow $p$-subgroup $P$ and $\pcore_{p'}(G)=1$. Then $|P:\ZB(P)|=p^2$ if and only if one of the following holds:
\begin{enumerate}[(A)]
\item\label{CaseA} $\pcore^{p'}(G)=\pcore_p(G)\times S$ where $\pcore_p(G)$ is non-abelian with $|\pcore_p(G):\ZB(\pcore_p(G))|=p^2$ and $S$ is a direct product of simple groups with abelian Sylow $p$-subgroups ($S=1$ allowed).

\item\label{CaseB} $\pcore^{p'}(G)=(\pcore_p(G)*C)\times S$ where $\pcore_p(G)$ is abelian, $S$ is a direct product of simple groups with abelian Sylow $p$-subgroups and $C$ is a quasisimple group with non-abelian Sylow $p$-subgroup of order $p^3$ and $|\ZB(C)|\le p$. 

\item\label{CaseC} $\F^*(G)=\pcore_p(G)\times S$ where $\pcore_p(G)$ is abelian, $S$ is a direct product of simple groups with abelian Sylow $p$-subgroups and $|G/\F^*(G)|_p=p$. There exists $x\in P\setminus\F^*(G)$ such that $|G:\CB_G(x)|_p=p$.

\item\label{CaseD} $p=2$ and $\F^*(G)=\pcore_2(G)\times S\times T$ where $\pcore_2(G)$ is abelian, $S$ is a direct product of simple groups with abelian Sylow $2$-subgroups and $T=A_7$ or $T=\PSL(2,p^{2f})$ where $p^{2f}\equiv 9\pmod{16}$. Moreover, $\pcore^{2'}(G)=P\F^*(G)$. There exists $x\in P\setminus\F^*(G)$ such that $|\CB_G(x)|_p=|G|_p$.
If $T=A_7$, then $x$ acts as a transposition on $T$ and if $T=\PSL(2,p^{2f})$, then $x$ acts as the field automorphism $x\mapsto x^{p^f}$ on $T$.
\end{enumerate}
Consequently, the property $|P:\ZB(P)|=p^2$ can be read off from the character table.
\end{thm}
\begin{proof}
In each step of the proof we make sure that the conclusion is detectable from the character table of $G$. 
By \autoref{TK}, we may assume that $|P'|=p$ (although this is a priori not visible from the character table).
Let $N:=\F^*(G)=\pcore_p(G)*\E(G)$ be the generalized Fitting subgroup. By \autoref{Kimmerle}, the character table determines $|N|$ and the non-abelian composition factors of $\E(G)$.

\textbf{Case~1:} $P\le N$.\\
Then $N=\pcore^{p'}(G)$. 
By comparing the minimal non-abelian subgroups of $G$ with $\E(G)$, the character table detects whether $\pcore_p(G)\cap\E(G)\ne 1$.

\textbf{Case~1.1:} $N=\pcore_p(G)\times\E(G)$.\\
Here $\E(G)$ is a direct product of simple groups and therefore the isomorphism type of $\E(G)$ is determined from the character table.
Going over to $G/\E(G)$, the character table tells us whether $\pcore_p(G)$ is abelian. If this is the case, then exactly one of the simple factors of $\E(G)$ has a non-abelian Sylow $p$-subgroup. We are in Case~\eqref{CaseB} by \autoref{lempodd}.
Now let $P\cap\E(G)$ be abelian. Then we may assume $\E(G)=1$. Here $N=P$ is the only Sylow $p$-subgroup of $G$ and $|\ZB(P)|=|\ZB(N)|$ is the number of $p$-elements $x\in G$ such that $|\CB_G(x)|_p=|P|$. Hence, $|P:\ZB(P)|$ is detected from the character table and we are in Case~\eqref{CaseA}.

\textbf{Case~1.2:} $Z=\pcore_p(G)\cap\E(G)\ne 1$.\\
There must be a component $C\le G$ with $Z\le\ZB(C)$. By \autoref{TK}, $C$ has a non-abelian Sylow $p$-subgroup $Q$, $|Z|=p$ and $|Q|=p^3$ by \autoref{lemp2} and \autoref{lempodd}.
If $C$ is not normal in $G$, then there is a conjugate component $C_1\le G$ and $C*C_1$ has an extraspecial Sylow $p$-subgroup $Q*Q_1$ of order $p^5$. Then $|Q*Q_1:\ZB(Q*Q_1)|=p^4$ contradicts $|P:\ZB(P)|=p^2$.
Hence, the character table should tell us that $C\unlhd G$ is the only component with $Z\le C$. We are in Case~\eqref{CaseB}.
Note that $P=(\pcore_p(G)*Q)\times P_1$ where $P_1$ is abelian. Now $|P:\ZB(P)|=p^2$ if and only if $\pcore_p(G)$ is abelian. This happens if and only if $|\CB_G(x)|_p=|P|$ for all $x\in\pcore_p(G)$. 

\textbf{Case~2:} $Q:=P\cap N<P$.\\
The character table detects whether $\E(G)$ has abelian Sylow $p$-subgroups since this can only happen if all components are (known) simple groups. 

\textbf{Case~2.1:} $Q_1:=Q\cap\E(G)$ is abelian.\\
Again by \autoref{TK}, $N=\pcore_p(G)\times\E(G)$. If $p>2$, then $\CB_P(Q)=\ZB(Q)$ and $\ZB(P)<Q$ by \autoref{GG}. If $p=2$ and $x\in\CB_P(Q)\subseteq\CB_G(Q_1)$, then by \autoref{GG} there exists $y\in\E(G)$ such that $xy\in\CB_G(\E(G))$. But then $xy\in\CB_G(N)\le N$ and we obtain $x\in P\cap N=Q$. Therefore, we have $\CB_P(Q)=\ZB(Q)$ and $\ZB(P)<Q$ independent of $p$. Hence, we may assume that $Q$ is abelian and $|P:Q|=p$. This is detected by the character table of $G/\E(G)$ using \autoref{lemcenter}. We are in Case~\eqref{CaseC}.
Let $x\in P\setminus Q$. Then $\ZB(P)=\CB_Q(x)$.  
If $x$ lies in the center of some Sylow $p$-subgroup $P_1$, then $x$ would centralize $P_1\cap N$ which was already excluded. Hence, $|P:\ZB(P)|=p^2$ if and only if $|G:\CB_G(x)|_p=p$.

\textbf{Case~2.2:} $Q_1$ is non-abelian.\\
Here $P=\CB_P(Q_1)Q_1$. This is only possible for $p=2$ by \autoref{GG}.
As in Case~1.2 there exists a unique component $C\le G$ with non-abelian Sylow $2$-subgroup $Q_2:=Q_1\cap C$. It follows from \autoref{lemp2} that $C\cong A_7$ or $C\cong\PSL(2,p^{2f})$ with $p^{2f}\equiv 9\pmod{16}$. In particular, $Q_2\cong D_8$ and $N=\pcore_2(G)\times C\times D$ where $D$ is a direct product of simple groups with abelian Sylow $2$-subgroups. We check with the character table that $G/C$ has abelian Sylow $2$-subgroups. We are in Case~\eqref{CaseD}.
Now we go over to $G/\pcore_2(G)D$. Then $|P|=16$. By \autoref{lem16}, $|P:\ZB(P)|=4$ if and only if $|P'|=2$. This is determined by the character table according to Theorem~A. Finally, observe that 
\[PN/N\le \CB_G(\pcore_2(G)D)N/N\le\Out(C)\cong C_2\times C_{2f}\]
(or $\Out(C)\cong C_2$ if $C=A_7$).
This shows that $NP=\pcore^{2'}(G)$.
\end{proof}

The example $G=S_4$ with $p=2$ in Case~\eqref{CaseC} shows that $|\pcore^{p'}(G)/N|$ is not necessarily $p$. The group $A_7\rtimes C_4$ with non-faithful action shows that $\pcore^{2'}(G)$ does not necessarily split over $N$ in Case~\eqref{CaseD}.

\end{document}